\documentclass[11pt,a4paper]{amsart}
\usepackage{amsmath}
\usepackage{amssymb}
\usepackage{hyperref}
\usepackage{cleveref}
\usepackage{mathrsfs}
\usepackage{tikz}
\usepackage{graphicx}
\graphicspath{{./figures/}}
\usepackage{a4wide}
\usetikzlibrary{calc}
\usepackage{tikz-cd}
\usepackage{adjustbox}
\usepackage{verbatim}
\usepackage{soul} 
\hypersetup{
	colorlinks=true,
	linkcolor=blue,
	citecolor=magenta,      
	urlcolor=cyan,
	pdftitle={Non-smoothable homeomorphisms of manifolds with boundary},
}

\newtheorem{thm}{Theorem}[section]
\newtheorem{prop}[thm]{Proposition}
\newtheorem{lem}[thm]{Lemma}
\newtheorem{cor}[thm]{Corollary}

\theoremstyle{definition}
\newtheorem{definition}[thm]{Definition}

\theoremstyle{remark}
\newtheorem{remark}[thm]{Remark}

\numberwithin{equation}{section}

\newcommand{\N}{\mathbb{N}}
\newcommand{\Z}{\mathbb{Z}}
\newcommand{\Q}{\mathbb{Q}}

\newcommand{\CP}{\mathbb{CP}}
\renewcommand{\SS}{\mathbb{S}}
\newcommand{\sm}{\setminus}

\newcommand{\ol}{\overline}
\newcommand{\genus}{\mathcal{G}}
\newcommand{\spinc}{{spin$^c$}}
\newcommand{\Spinc}{{\mathrm{Spin}^c}}

\DeclareMathOperator\Torsion{Torsion}

\DeclareMathOperator\Id{Id}

\DeclareMathOperator\Homeo{Homeo}

\DeclareMathOperator\Diff{Diff}

\DeclareMathOperator\PD{PD}

\DeclareMathOperator\ev{ev}

\DeclareMathOperator\Image{im}

\DeclareMathOperator\Tor{Tor}
\DeclareMathOperator\Aut{Aut}
\DeclareMathOperator\ad{ad}
\DeclareMathOperator\tb{tb}

\title[Non-smoothable homeomorphisms of $4$-manifolds with boundary]{Non-smoothable homeomorphisms of $4$-manifolds with boundary}
\author{Daniel Galvin}
\address{Max Planck Institut f\"{u}r Mathematik, Vivatsgasse 7, 53111 Bonn, Germany}
\email[Daniel A.P. Galvin]{galvin@mpim-bonn.mpg.de}
\author{Roberto Ladu}
\address{ Fakult\"{a}t f\"{u}r Mathematik, Ruhr Universit\"{a}t Bochum,  Universit\"{a}tsstrasse 150, D-44801 Bochum, Germany}
\email[Roberto Ladu]{roberto.ladu.math@gmail.com}

\begin{document}

\maketitle

\begin{abstract}
We construct the first examples of non-smoothable self-homeomorphisms of smooth $4$-manifolds with boundary that fix the boundary and act trivially on homology.
As a corollary, we construct self-diffeomorphisms of $4$-manifolds with boundary that fix the boundary and act
trivially on homology but cannot be isotoped to any self-diffeomorphism supported in a collar of the boundary and, in particular, are not isotopic to any generalised Dehn twist.
\end{abstract}

\section{Introduction}

\subsection{Results}

Let $X$ be a smooth, compact, oriented $4$-manifold with boundary.  We will denote by $\Homeo^+(X,\partial X)$ the topological group of orientation-preserving self-homeomorphisms of $X$ that restrict to the identity map on $\partial X$, topologised using the compact-open topology. 
We say that a homeomorphism $f\in\Homeo^+(X,\partial X)$ is \emph{non-smoothable} if it is not isotopic \emph{relative to the boundary} to any self-diffeomorphism of $X$.  

We denote by $\Tor(X,\partial X)\subset \pi_0\Homeo^+(X,\partial X)$ the \emph{$($topological$)$ Torelli group} of $(X,\partial X)$, the subgroup of isotopy classes of homeomorphisms that induce the identity map on $H_2(X)$. 
When $X$ is simply-connected and \emph{closed},  Perron--Quinn \cite{perron_1986,quinn_1986, gabai2023pseudoisotopiessimplyconnected4manifolds} showed that $\Tor(X,\partial X)$ is trivial and hence all of its elements are smoothable. However, if $\partial X\neq \emptyset$ then Orson--Powell \cite{orson_powell_2023} showed that it is in general non-trivial and hence could contain non-smoothable homeomorphisms. Our first result shows the existence of such non-smoothable elements of~$\Tor(X,\partial X)$.

\begin{thm}\label{thm:1}
   There exists an infinite family of pairwise non-diffeomorphic compact, oriented, smooth, simply-connected $4$-manifolds $\{(X_n,\partial X_n)\}_{n\in\N}$ with connected boundary and $\Tor(X_n,\partial X_n)$ of infinite order such that, for each $n$, all non-trivial elements in $\Tor(X_n,\partial X_n)$ are non-smoothable.
\end{thm}
    In fact, we construct two separate such families, one such that the boundaries $\partial X_n$ are pairwise non-diffeomorphic (\Cref{thm:FamilyX_n}) and another family $\{Z_n\}_{n\in\N}$ such that the boundaries $\partial Z_n$ are all diffeomorphic and the $Z_n$ are all homeomorphic relative to their boundaries (\Cref{thm:FamilyZ_n}).  Furthermore, the first of these families is minimal in the sense that the produced manifolds have the simplest possible intersection forms.  See \Cref{rem:minimality} for more details.
    
    We also note that the homeomorphisms of \Cref{thm:1} are not isotopic to any diffeomorphism even `absolutely', i.e.\ when considering isotopies that do not fix the boundary pointwise.  Indeed, we will see in \Cref{sec:gen_dehn_twists} that relative and absolute non-smoothability are equivalent notions for $4$-dimensional manifolds.

There is an interesting class of maps in $ \Tor(X,\partial X)$ represented by a smooth map.  Given a loop $\gamma$ of orientation-preserving diffeomorphisms of the boundary based at the identity, we can form a diffeomorphism $\varphi_{\gamma}\colon(X,\partial X)\to (X,\partial X)$ by inserting $\gamma$ into a collar of the boundary and extending via the identity map.  Such diffeomorphisms are called \emph{generalised Dehn twists}, and, since they are supported on a collar of the boundary, they give rise to elements in $\Tor(X,\partial X)$ that can be represented by smooth maps.  It is an interesting question whether a given smoothable element of $\Tor(X,\partial X)$ is realised by a generalised Dehn twist.

Our second result shows the non-realisability of smoothable elements of $\Tor(X,\partial X)$ by generalised Dehn twists.

\begin{thm}\label{thm:2}
    There exists an infinite family of pairwise non-diffeomorphic compact, oriented, smooth, simply-connected $4$-manifolds $\{(W_n,\partial W_n)\}_{n\in \N}$ with connected boundary and $\Tor(W_n,\partial W_n)$ of infinite order, such that, for each $n$, all mapping classes in $\Tor(W_n,\partial W_n)$ are smoothable, but only the identity map is supported on a collar of the boundary and, in particular, only the identity map is realised by a generalised Dehn twist.
\end{thm}

\subsection{Background}
The question about smoothable versus non-smoothable homeomorphisms for closed, oriented, simply-connected $4$-manifolds has been studied extensively.  If $X$ is such a manifold (or has boundary a homology sphere) with indefinite intersection form or the rank of $H_2(X)$  at most $8$, then Wall \cite{wall_1964} showed that all isometries of the intersection form $\Aut(H_2(X\# (\SS^2\times \SS^2)),\lambda_{X\# (\SS^2\times \SS^2)})$ can be realised by diffeomorphisms (and hence all self-homeomorphisms of $X\# (\SS^2\times \SS^2)$ are smoothable by Perron--Quinn \cite{perron_1986, quinn_1986, gabai2023pseudoisotopiessimplyconnected4manifolds}).  Ruberman and Strle \cite{ruberman_strle_2023} extended this result to show that any self-homeomorphism of $X\# (\SS^2\times \SS^2)$ that acts trivially on the homology of the $\SS^2\times \SS^2$-summand is smoothable.

Conversely, for closed $4$-manifolds, Friedman and Morgan constructed the first examples of non-smoothable homeomorphisms by considering self-homeomorphisms of Dolgachev surfaces \cite{friedman_morgan_1988}.  The general argument for producing such non-smoothable homeomorphisms goes in the following manner.  Firstly, by Freedman \cite{freedman_1982}, we know that any automorphism of the intersection form is realisable by a homeomorphism.  Then one uses a gauge-theoretic invariant (e.g. Seiberg-Witten invariants) to show that certain automorphisms of the intersection form are not realisable by a diffeomorphism, since diffeomorphisms must preserve certain homology classes (e.g. Seiberg-Witten basic classes).  This style of argument has been used to produce many more examples.  In particular, Donaldson \cite{donaldson_1990} showed that the $K3$ surface admits a non-smoothable homeomorphism.  Further instances are known, see \cite{morgan_szabo_1997}, \cite{baraglia_2021}.

There is a  natural generalisation of this idea to the case of $4$-manifolds with boundary using  Monopole or Heegaard-Floer homology \cite{KM}\cite{OzsvathSzaboHolomorphicTriangles} which consists of looking
at the cobordism maps induced in Floer homology by the $4$-manifold together with
a \spinc-structure. Indeed, two \spinc-structures related by a  diffeomorphism fixing the boundary induce the same cobordism map up to multiplication by $\pm 1$, depending on the action of the diffeomorphism on homology.
However, this approach cannot obstruct the smoothability of homeomorphisms in the Torelli group because, as we will see in \Cref{sec:variations}, such homeomorphisms preserve the relative isomorphism class of a \spinc-structure since they act trivially on $H_2(X,\partial X)$.

A different approach, based on Seiberg-Witten Floer stable homotopy type \cite{ManolescuSWFloerHomotopy},  has been used recently by Konno and Taniguchi \cite[Thm 1.7]{konno_taniguchi_2022} to construct non-smoothable homeomorphisms for a large class of $4$-manifolds with boundary a rational homology sphere. The technical results  underpinning this approach require $b_1(\partial X) = 0$ and therefore this approach cannot be directly applied to find non-smoothable elements in the Torelli group because the latter is non-trivial only when $b_1(\partial X)\geq 2$ (see \Cref{thm:orson_powell_variations}). 
An enhancement of this approach to the case $b_1(\partial X)>0$ might become available in the future using generalisations of \cite{ManolescuSWFloerHomotopy}, e.g. \cite{KhandhawitLinSasahira},\cite{SasahiraStoffregen}.  Regardless, in this paper we take a different approach based on embedding $X$ into a closed $4$-manifold (see \Cref{sbs:variations_closed_manifolds}).

\subsection{Outline}

We briefly outline the contents of the paper.  In \Cref{sec:variations} we recall the classification of $\Tor(X,\partial X)$ in terms of algebraic objects called variations, and prove a key proposition (\Cref{lem:gluing_lemma}) which we will use to detect elements of the Torelli group.  In \Cref{sec:proof_of_existence_condition} we prove technical conditions under which we can guarantee the existence of non-smoothable elements of the Torelli group.  In \Cref{sec:examples} we use the conditions from the previous section to produce our two infinite families of examples and hence prove \Cref{thm:1}.  Finally, in \Cref{sec:gen_dehn_twists} we consider generalised Dehn twists and prove \Cref{thm:2}.

\subsection{Acknowledgements}
Both authors would like to thank Mark Powell for bringing this problem to their attention,  answering their many questions and for his insightful comments,
and Simona Vesel\'{a} and Burak \"{O}zba\u{g}c{\i} for their suggestions and careful review of previous versions of this paper.  Both authors are also grateful to the anonymous referee for their very helpful comments.
The second author would like to thank Paolo Ghiggini, Arunima Ray, Steven Sivek and Peter Teichner for helpful conversations and encouragement.  
The first author was supported financially by the University of Glasgow.  The second author is thankful to the Max Planck Institute for Mathematics in Bonn for hospitality and financial support.

\section{Variations}\label{sec:variations}

\subsection{Definitions}

The aim of this section is to describe the classification of homeomorphisms up to isotopy for simply-connected, topological, oriented $4$-manifolds with boundary.  This classification is due to the work of Osamu Saeki, Patrick Orson, and Mark Powell \cite{saeki_2006,orson_powell_2023}.  We will begin by defining what a variation is, which is the central object involved in the classification.  Unlike the rest of this paper, all of the statements in this section are purely topological in nature, and so hold regardless of whether the manifolds in question are smooth or topological.

\begin{definition}\label{def:homeo_to_variation}
    Let $X$ be a simply-connected, oriented $4$-manifold with boundary and let $f\in\Homeo^+(X,\partial X)$ be an orientation-preserving homeomorphism relative to the boundary $\partial X$.  Then the \emph{variation induced by} $f$, denoted as $\Delta_f$, is defined as
    \begin{align*}
        \Delta_f\colon H_2(X,\partial X) &\to H_2(X) \\
        [\Sigma] &\mapsto [\Sigma-f(\Sigma)],
    \end{align*}
    where $\Sigma$ denotes a relative $2$-chain.  Note that the homology class $[\Sigma-f(\Sigma)]$ does not depend on the choice of representative relative $2$-chain $\Sigma$ \cite[Sec.~2.2]{orson_powell_2023}.
\end{definition}

We can also define variations without reference to a homeomorphism.

\begin{definition}
    Let $X$ be a simply-connected, oriented $4$-manifold with boundary and let $\Delta\colon H_2(X,\partial X)\to H_2(X)$ be a homomorphism.  Then we say that $\Delta$ is a \emph{Poincar\'{e} variation} if 
    \[
    \Delta+\Delta^! = \Delta\circ q_* \circ \Delta^!\colon H_2(X,\partial X)\to H_2(X),
    \]
    where $q$ is the inclusion map of pairs $(X,\emptyset) \to (X,\partial X)$ and $\Delta^!$ denotes the `umkehr' homomorphism to $\Delta$, defined as the following composition:
    \[
    \Delta^!\colon H_2(X,\partial X) \xrightarrow{\PD^{-1}} H^2(X)\xrightarrow{\ev} H_2(X)^*\xrightarrow{\Delta^*} H_2(X,\partial X)^* \xrightarrow{\ev^{-1}} H^2(X,\partial X) \xrightarrow{\PD} H_2(X).
    \]
\end{definition}

Following the notation of Orson-Powell, we will denote the set of Poincar\'{e} variations of $(X,\partial X)$ as $\mathcal{V}(H_2(X),\lambda_X)$, where $\lambda_X$ denotes the intersection form of $X$.  This notation is used because it is shown in \cite[Sec.~7]{orson_powell_2023} that the set of variations only depends on the isometry class of the intersection form $(H_2(X),\lambda_X)$, rather than on the $4$-manifold specifically.

We can give $\mathcal{V}(H_2(X),\lambda_X)$ the structure of a group due to the following lemma of Saeki.

\begin{lem}[{\cite[Lem.~3.5]{saeki_2006}}]\label{lem:variation_group}
    The set $\mathcal{V}(H_2(X),\lambda_X)$ forms a group with multiplication given by
    \[
    \Delta_1 \cdot \Delta_2 := \Delta_1 +(\Id-\Delta_1\circ q_*)\circ\Delta_2,
    \]
    identity the zero homomorphism, and inverse given by
    \[
    \Delta^{-1}=-(\Id-\Delta \circ q_*)\circ \Delta.
    \]
\end{lem}

Further, we have 
\begin{lem}[{\cite[Lem.~3.2]{saeki_2006}}]\label{lem:homeo_variations_poincare_variations}
    Let $X$ be a compact, simply-connected, oriented, topological $4$-manifold with boundary $\partial X$ and let $f\in\Homeo^+(X,\partial X)$.  Then $\Delta_f$ is a Poincar\'{e} variation.
\end{lem}

The converse of the above result, that all Poincar\'{e} variations are induced via homeomorphisms, is given by \cite[Thm.~A]{orson_powell_2023}.

\subsection{The Torelli group}

The map which sends a homeomorphism to its variation gives a factorisation of the map which takes the induced automorphism of the form for a homeomorphism:
\begin{equation}\label{eq:induced_auto_factorisation}
    \pi_0\Homeo^+(X,\partial X) \xrightarrow{f\mapsto \Delta_f} \mathcal{V}(H_2(X),\lambda_X) \xrightarrow{\Delta\mapsto \Id-\Delta\circ q_*} \Aut(H_2(X),\lambda_X).
\end{equation}
It is the result of Freedman--Perron--Quinn \cite{freedman_1982, perron_1986, quinn_1986} that, for a \emph{closed}, simply-connected $4$-manifold $X$,  the above composition is a bijection.  Hence, all homeomorphisms of a closed simply-connected $4$-manifold that map to the trivial element of $\Aut(H_2(X),\lambda_X)$ are isotopic to the identity map.  For manifolds with non-empty boundary, the classification is more subtle \cite[Thm.~A]{orson_powell_2023}, and in particular we can have homeomorphisms that are not isotopic to the identity but still induce the trivial element of $\Aut(H_2(X),\lambda_X)$ under \ref{eq:induced_auto_factorisation}.  

\begin{definition}
    Let $X$ be a compact, simply-connected, oriented, $4$-manifold with boundary $\partial X$.  We define the \emph{Torelli group} $\Tor(X,\partial X)\subset \pi_0\Homeo^+(X,\partial X)$ to be the subgroup of homeomorphisms that induce the trivial element of $\Aut(H_2(X),\lambda_X)$ under \ref{eq:induced_auto_factorisation}.
\end{definition}

Note that the subgroup of variations which are induced by elements in the Torelli group is exactly the subgroup of variations satisfying that $\Delta\circ q_*\colon H_2(X)\to H_2(X)$ is the zero map.  For such Poincar\'{e} variations, we can construct a skew-symmetric pairing in the following way.  Let $\Delta$ be a Poincar\'{e} variation.  Then this gives rise to a map \[(\eta_\Delta)^{\ad} \colon H_1(\partial X)\to H_2(\partial X)\cong H_1(\partial X)^*\] (note that the last isomorphism is given by Poincar\'{e} duality and universal coefficients) by first lifting an element in $H_1(\partial X)$ to an element in $H_2(X,\partial X)$, mapping to $H_2(X)$ using $\Delta$ (note that this image does not depend on the choice of lift) and then noting that by the definition of Poincar\'{e} variations, this element lifts uniquely to an element in $H_2(\partial X)$.  As suggested by the notation, we can interpret this map as the adjoint of a pairing:
\[
\eta_{\Delta}\colon H_1(\partial X)\times H_1(\partial X)\to \Z
\]
and it is stated in \cite[Prop.~4.2]{saeki_2006} that this form is skew-symmetric.  To see this, it is enough to verify that $\eta_{\Delta}^{\ad}(x)(x)=0$ for any $x \in H_1(\partial X)$, and this fact is geometrically clear from the definition of $\eta^{\ad}$.  More crucially, we can go the other way.  Let $\eta\colon H_1(\partial X)\times H_1(\partial X)\to \Z$ be a skew-symmetric pairing.  Then we can define a variation $\Delta_{\eta}$ as the following composition:
\begin{equation}\label{eq:pairing_to_variation}
    H_2(X,\partial X)\xrightarrow{\partial} H_1(\partial X)\xrightarrow{\eta^{\ad}} H_1(\partial X)^* \xrightarrow{\mathrm{ev}^{-1}} H^1(X)\xrightarrow{PD} H_2(\partial X)\xrightarrow{i_*} H_2(X)
\end{equation}
where the map $\partial$ denotes the connecting homomorphism in the long exact sequence of the pair and $i\colon \partial X\to X$ denotes the inclusion.  We have the following sequence, due to Saeki.

\begin{prop}[{\cite[Prop.~4.2]{saeki_2006}},{\cite[Thm.~7.13]{orson_powell_2023}}]
    Let $X$ be a compact, simply-connected, oriented, topological $4$-manifold with connected boundary $\partial X$.  Then the following is a short exact sequence:
    \[
    0\to \Lambda^2H_1(\partial X)^* \xrightarrow{} \mathcal{V}(H_2(X),\lambda_X)\xrightarrow{} \Aut(H_2(X),\lambda_X)\to 0.
    \]
\end{prop}

So it follows that, in the connected boundary case, we have that the variations which induce the trivial map on homology are in one-to-one correspondence with skew-symmetric forms on $H_1(\partial X)$.

We have the following classification of the Torelli group, due to Orson-Powell.

\begin{thm}[{\cite[Cor. D]{orson_powell_2023}}]\label{thm:orson_powell_variations}
    Let $X$ be a compact, simply-connected, oriented $4$-manifold with connected boundary $\partial X$.  Then there is an isomorphism of groups:
    \begin{align*}
        \Tor(X,\partial X)&\xrightarrow{\cong} \Lambda^2H_1(\partial X)^*, \\
        [f] &\mapsto \eta_{\Delta_f}.
    \end{align*}
\end{thm}

\begin{remark}\label{rem:minimality}
    It follows immediately from this that $\Tor(X,\partial X)$ is non-trivial if and only if $b_1(\partial X) \geq 2$.  In fact, since $X$ is simply-connected, we can say more.  We claim that 
      $b_2(X)\geq b_1(\partial X)$ with equality holding if and only if $X$ has vanishing intersection form.  To see this, consider the exact sequence
    \[
    0\to H_2(\partial X;\Q) \to H_2(X;\Q) \xrightarrow{\lambda^{\ad}}H_2(X;\Q)^* \to H_1(\partial X;\Q) \to H_1(X;\Q)=0,
    \]
    where $\lambda^{\ad}$ is the adjoint of the intersection form and 
    the penultimate map is the composition of the inverse of the evaluation map, Poincar\'e duality and the connecting morphism of the long exact sequence of the pair.
    Now the claim is implied by the surjectivity of the penultimate map and the exactness of the sequence in $H_2(X;\Q)^*$.
    It follows that examples of $4$-manifolds with non-trivial $\Tor(X,\partial X)$  must have $b_2(X)\geq 2$.
\end{remark}

\subsection{Applying variations to closed manifolds}\label{sbs:variations_closed_manifolds}

Let $W$ be a simply-connected, oriented $4$-manifold with connected boundary $\partial W$.  In \Cref{sec:proof_of_existence_condition} we will want to use variations to prove that a homeomorphism $f\colon (W,\partial W)\to (W,\partial W)$ is non-smoothable.  In doing so, we will need the following lemma.

\begin{prop}\label{lem:gluing_lemma} 
    Let $W_1$ be a simply-connected, oriented $4$-manifold with connected boundary $\partial W_1\cong Y$, $W_2$ be an oriented $4$-manifold with  boundary $\partial W_2\cong -Y$ with $H^1(W_2;\Z)=0$ and let $X:= W_1\cup_{Y} W_2$ be the closed, oriented union.  Let $\eta\colon H_1(\partial W_1)\times H_1(\partial W_1)\to \Z$ be a skew-symmetric pairing, denote by $\Delta_\eta$ the induced variation $($given by \Cref{eq:pairing_to_variation}$)$ and denote by $\varphi_\eta\colon W_1\to W_1$ the induced homeomorphism $($given by \Cref{thm:orson_powell_variations}$)$.  Consider the umkehr map to the inclusion $i_1\colon W_1\to X$,
    \[
    i_1^!\colon H_2(X) \xrightarrow{\PD^{-1}} H^2(X) \xrightarrow{i_1^*} H^2(W_1) \xrightarrow{\PD} H_2(W_1, Y).
    \]Then for any class $x\in H_2(X)$ we have that
    \begin{equation}\label{eq:gluing_lemma}
        (\varphi_\eta \cup \Id_{W_2})_*(x) = x - (i_1)_*\Delta_\eta(i_1^!(x))\in H_2(X),
    \end{equation}
    where $\varphi_\eta \cup \Id_{W_2}\colon X\to X$ is the homeomorphism defined as $\varphi_\eta$ on $W_1$ and as $\Id_{W_2}$ on $W_2$.
\end{prop}

\begin{proof} 
    Let $\Sigma\subset X$ be an embedded, closed, oriented surface representing $x$, transverse to $Y$ (for topological transversality see \cite[Thm.~9.5A]{freedman_quinn_1990}).  
    
   The statement $i_1^!(x) = [\Sigma\cap W_1] \in H_2(W_1,Y)$ is equivalent to the commutativity of the following diagram:
    \[
    \begin{tikzcd}
        H_2(X) \arrow[r,"q_*"] & H_2(X,W_2) & H_2(W_1,Y) \arrow[l,"\cong"'] \\
        H^2(X) \arrow[r,"i_1^*"] \arrow[u,"\PD", "\cong"'] & H^2(W_1) \arrow[ru, "\PD"', "\cong"] & 
    \end{tikzcd}
    \]
    where $q_*$ is the map induced by the inclusion $q\colon (X,\emptyset)\to (X,W_2)$
    which sends $[\Sigma]$ to $[\Sigma\cap W_1]$  and the written isomorphism $H_2(W_1,Y)\to H_2(X,W_2)$ comes from excision.   
    
    We now prove the commutativity of the diagram.  Let $y\in H^2(X)$.  Going first to the right, this maps to $i_1^*(y)\frown [W_1,Y]\in H_2(W_1,Y)$ and then to $y\frown (i_1)_*[W_1,Y]\in H_2(X,W_2)$ by naturality of the (relative) cap product.  Going the other way, $y\in H^2(X)$ is mapped to $q_*(y\frown [X])=y\frown q_*[X]\in H_2(X,W_2)$, again by naturality of the (relative) cap product.  It remains to be shown that these are equal, i.e.\ that $q_*[X]=(i_1)_*[W_1,Y]$.  By considering the long exact sequence of the pair $(X,W_2)$ one sees that the map $q_*\colon H_4(X)\to H_4(X,W_2)$ is an isomorphism, which means that $q_*[X]$ is a generator of $H_4(X,W_2)$ (to see this one needs to use the fact that $H_3(W_2)=0$, which follows from Poincar\'{e} duality and universal coefficients theorem, using the fact that $H^1(W_2;\Z)=0$ and that $W_2$ has connected boundary).  Similarly, by excision, $(i_1)_*\colon H_4(W_1,Y)\to H_4(X,W_2)$ is an isomorphism, so it follows that $(i_1)_*[W_1,Y]$ is a generator of $H_4(X,W_2)$.  The fact that they are the same generator follows from the orientations on $X, W_1$ and $W_2$ having been chosen to be compatible.  This completes the proof that the diagram commutes.
    
    As a singular $2$-chain, $\Sigma = (\Sigma\cap W_1) + (\Sigma\cap W_2) \in C_2(X)$. Similarly, 
    the singular $2$-chain induced by $(\varphi_\eta \cup \Id_{W_2})(\Sigma)$, i.e.\ the left hand side of \ref{eq:gluing_lemma}, is equal to the sum  $\varphi_\eta(\Sigma\cap W_1) + (\Sigma\cap W_2)$ in $C_2(X)$. Hence we have that:
    \begin{equation*}
        \Sigma - (\varphi_\eta \cup \Id_{W_2})(\Sigma)  =  (\Sigma\cap W_1)- \varphi_\eta(\Sigma\cap W_1)  \in C_2(X).
    \end{equation*}
    The right hand side is homologous to the cycle induced by the glued-up surface $(\Sigma\cap W_1)\cup -{\varphi_\eta(\Sigma \cap W_1)}$ which is equal to $(i_1)_*\Delta_\eta ([\Sigma \cap W_1])$ by \Cref{def:homeo_to_variation}.
\end{proof}

\section{Sufficient conditions for non-smoothability}\label{sec:proof_of_existence_condition}

\newcommand{\cI}{\mathcal{I}}
\newcommand{\cB}{\mathcal{B}}
\newcommand{\sstruc}{\mathfrak{s}}
In this section, all manifolds will be considered to be smooth.
For any closed, oriented, $4$-manifold  $X$, we will denote by $\Spinc(X)$ the set of isomorphism classes of \spinc-structures on $X$, and by  $\cI(X,\cdot)\colon \Spinc(X)\to \mathcal{Y}$ a map taking values in an abelian group $\mathcal{Y}$ such that the action of  $ \mathrm{Diffeo}^+(X)$ on $H_2(X)$ by pull-back preserves the set of $\cI$-basic classes, defined as
\begin{equation*}
    \cB_\cI(X) := \{c_1(\sstruc)\in H^2(X) \ | \  \cI(X, \sstruc)\neq 0, \sstruc \in \Spinc(X)\},
\end{equation*}
and moreover this  set is \emph{finite}.
For example, if  $b_2^+(X)\geq 2$, and $\sstruc \in \Spinc(X)$, $\cI(X,\sstruc)$ may be taken to be the Seiberg-Witten invariant $SW(X,\sstruc)$ \cite{witten_1994}, the Ozsv\'ath-Szab\'o mixed invariant $\Phi_{X,\sstruc}$ \cite{OzsvathSzaboHolomorphicTriangles},
or the Bauer-Furuta invariant $BF_{X,\sstruc}$ \cite{BauerFuruta}.
The finiteness of $BF$-basic classes is not stated explicitly in \cite{BauerFuruta}, but can be proved using curvature inequalities as observed in the proof of \cite[Thm. 4.5]{ManolescuMarengonPiccirillo}.

We now prove the main proposition that we will use to detect non-smoothability for homeomorphisms in the Torelli group.

\begin{prop}\label{Lemma:KeyLemma} Let $W^4$ be a compact, oriented $4$-manifold with connected boundary $Y$.
Suppose that $\pi_1(W) = 1$ and that $b_1(Y)\geq 2$.
If  $W$ embeds in a closed, oriented $4$-manifold $X$ such that 
for some  $\sstruc\in \Spinc(X)$, 
    \begin{enumerate}
        \item $\cI(X,\sstruc)\neq 0$,
        \item $i^*_{Y,X}(c_1(\sstruc))\in H^2(Y)$ is non-torsion where $i_{Y,X}:Y\hookrightarrow X$  is the inclusion,
        \item $H^1(X\setminus W) = 0$,
    \end{enumerate} then there exists infinitely many non-smoothable 
mapping classes in $\Tor(W,Y)$. If in addition $b_1(Y) = 2$  then any non-trivial element of $\Tor(W,Y)$ is non-smoothable.
\end{prop}

\begin{proof} Let $\sstruc$ be a Spin$^c$ structure on $X$ satisfying the hypothesis of the theorem. To avoid clutter, it is convenient to denote $\zeta_X := c_1(\sstruc)$ and its restrictions by $\zeta_W := i^*_{W,X}c_1(\sstruc)$ and $\zeta_Y := i^*_{Y,X}c_1(\sstruc)$.

Since $\zeta_Y\in H^2(Y;\Z)$ is not torsion $\PD(\zeta_Y) = dv_1$, for some $d\in \Z\setminus \{0\}$ and an indivisible element $v_1\in H_1(Y;\Z)$.
    Extend $v_1$ to $v_1,\dots, v_{b_1(Y)} \in H_1(Y)$, a lift of a basis of $H_1(Y)/\mathrm{Torsion}(H_1(Y))$.
    Now we set $\eta := v_1^*\wedge v_2^* \in \Lambda^2 (H_1(Y)^*)$  where $v_i^*$ denotes the Hom dual with respect to the above basis  (note that $v_2\neq 0$ exists since we assumed $b_1(Y)\geq 2$).
    
    By \Cref{thm:orson_powell_variations}, for each $k\in \Z\setminus \{0\}$, there is a unique mapping class in  $\Tor(W,Y)$ associated to $k\eta$, and we define $\varphi_k \in \Homeo^+(W,Y)$ to be an arbitrary representative of that class.
    By construction, each $\varphi_k$ acts trivially on $H_2(W)$. The rest of the proof is devoted to showing that $\varphi_k$ is non-smoothable for infinitely many values of $k$.
    
    For each $k$,  we define $ \hat \varphi_k:= \varphi_k\cup\Id_{X\sm \mathrm{int}(W)} \in \Homeo^+(X)$ as in \Cref{lem:gluing_lemma} to be the homeomorphism obtained by extending $\varphi_k$
    as the identity on $X\sm W$.
    The non-smoothability of $\hat \varphi_k$ for infinitely many $k$, which we are now going to prove, implies the analogous statement 
    for~$\varphi_k$.

    We will show  that $\{{(\hat \varphi_k)}_* \PD(\zeta_X)\}_{k\in \Z\setminus \{0\}}$ is infinite, which is equivalent to $\{\hat \varphi_k^* \zeta_X\}_{k\in \Z\setminus \{0\}}$ being infinite. Since $\cB_\cI(X)$ is finite 
    and is preserved by the action of $\Diff^+(X)$, this will imply  that  $\hat \varphi_k^*$ is non-smoothable for infinitely many $k$.

    It follows from  \Cref{lem:gluing_lemma} that
    \begin{equation}\label{eq:eq1}
        (\hat \varphi_k)_* (\PD (\zeta_X)) = \PD (\zeta_X) - (i_{W,X})_* \circ \Delta_{\varphi_k}(\PD (\zeta_W)) \in H_2(X).
    \end{equation}

    From \eqref{eq:pairing_to_variation} we see that 
    \begin{equation*}
       (i_{W,X})_*\circ \Delta_{\varphi_k}(\PD(\zeta_W)) = k \cdot  (i_{Y,X})_*\circ \PD \circ \mathrm{ev}^{-1} \circ \eta^{ad}\circ \partial \circ \PD (\zeta_W) \in H_2(X).
    \end{equation*}
    We claim that the right hand side is equal to a non-torsion element times
    $k$. This will imply the desired result by \eqref{eq:eq1}. 
    We have that 
    \begin{equation*}
    \begin{split}
         (i_{W,X})_*\circ \Delta_{\varphi_k}(\PD(\zeta_W)) &= k \cdot  (i_{Y,X})_*\circ \PD \circ \mathrm{ev}^{-1} \circ \eta^{ad}\circ \partial \circ \PD (\zeta_W)\\
       &= k \cdot  (i_{Y,X})_*\circ \PD \circ \mathrm{ev}^{-1} \circ \eta^{ad}\circ\PD (\zeta_Y)\\
       & = k \cdot  (i_{Y,X})_*\circ \PD \circ \mathrm{ev}^{-1} \circ \eta^{ad}(d v_1)\\
       &= k d \cdot  (i_{Y,X})_*\circ \PD \circ \mathrm{ev}^{-1} (v_2^*).
    \end{split}
    \end{equation*}
        
    
    Since $v_2$ is non-torsion and the maps $\mathrm{ev}^{-1}$ and $\PD$ are isomorphisms, the claim will follow if we prove that $(i_{Y,X})_*\colon H_2(Y)\to H_2(X)$ is injective.
    Now $(i_{Y,X})_* =(i_{W,X})_*\circ (i_{Y,W})_* $.
    The kernel of $(i_{Y,W})_*$ is equal to the  image of $H_3(W,Y)\to H_2(Y)$ in the long exact sequence of the pair, which is trivial since $H_3(W,Y)\cong H^1(W) = 0$.
    Similarly the kernel of $(i_{W,X})_*$ is equal to the  image of $H_3(X,W)\to H_2(W)$ which is zero because $H_3(X,W)\cong H_3(X\setminus \mathrm{int}(W), Y)$ by excision and by assumption $0= H^1(X\setminus W) \cong H_3(X\setminus \mathrm{int}(W), Y)$. Being the composition of injective maps, $i_{Y,W}$ is injective.  This completes the proof that there are infinitely many non-smoothable mapping classes in $\Tor(W,Y)$.

    To prove the last statement we assume now that $b_1(Y) = 2$.
    Then, under the isomorphism from \Cref{thm:orson_powell_variations}, we can identify $\Tor(W,Y)$ with the infinite cyclic group generated by~$\eta$. 
 Above we showed that there exists $k_0>0$ such that  $\varphi_{k\eta}$
    is non-smoothable for any~$|k|>~k_0$. The non-smoothability of $\Tor(X,Y)\setminus \{\Id_X\}$ follows from this by using the
    fact that smoothable mapping classes form a subgroup of $\Tor(X,Y)$ and that all non-trivial
    subgroups of $\Z$ are infinite.
\end{proof}

\Cref{Lemma:KeyLemma} has an immediate application to symplectic fillings.  For the reader's convenience we recall that a strong symplectic filling of a contact $3$-manifold $(Y,\xi)$ is a compact, symplectic  $4$-manifold $(W,\omega)$ with oriented boundary $Y$
such that such that there exists a Liouville vector field $V$ defined in a neighbourhood of $\partial W$ pointing outwards along $Y$  and such that  the pull-back of the $1$-form $\omega(V,\cdot)$ to $Y$  induces the contact structure $\xi$ on $Y$. The interested reader is referred to \cite{geiges_book, ozbagci_stipsicz_boook}.

\begin{cor}\label{cor:existence_condition}
    Let $(W,\omega)$ be a strong symplectic filling of  $(Y,\xi)$. Further suppose that $W$ is simply-connected, $b_1(Y)\geq 2$ and that  $c_1(\xi)\in H^2(Y)$  is not torsion. Then the same conclusions of \Cref{Lemma:KeyLemma} hold.
\end{cor}

\begin{proof} It is possible to embed $(W,\omega)$ symplectically 
into a closed symplectic $4$-manifold $(X,\omega_X)$ \cite{eliashberg_2004,EtnyreSymplecticFillings} (see also \cite{LiscaMatic} for the Stein case).  Furthermore, we can arrange 
that $\pi_1(X\setminus W) = 1$ and that $b_2^+(X)\geq 2$  \cite[Sec. 6]{etnyre_min_mukherjee_2022} (note that the symplectic cap called $X$ in \cite[Sec. 6]{etnyre_min_mukherjee_2022} plays the role of $X\setminus \mathrm{int}(W)$ in our proof).
    Being symplectic, it follows from Taubes' work \cite{taubes_1994} that $c_1(X)$ is a Seiberg-Witten basic class.
    Because the embedding is symplectic, we have that $i_{Y,X}^*c_1(X) = i_{Y,W}^*\circ i_{W,X}^*c_1(X) = c_1(\xi)$, which is non-torsion by assumption.
    Now apply \Cref{Lemma:KeyLemma}.\end{proof}

\section{Constructing examples}\label{sec:examples}
    In this section we will construct two infinite families of $4$-manifolds, each supporting an infinite family of non-smoothable mapping classes in their Torelli groups.

    \subsection{Family from Legendrian surgery}
    
    Stein domains (see \cite[Sec. 11.2]{gompf_stipsicz_1999} for an introduction) are a particular case of symplectic manifolds which are also strong symplectic fillings of their boundary \cite[Prop. 5.4.9]{geiges_book}. A $4$-manifold can be given the structure of a Stein domain if and only if it can be described by a special type of Kirby diagram
    \cite{Gompf_HandlebodyConstructionStein}:
     a \emph{Legendrian link diagram in standard form} \cite[Def. 11.1.7]{gompf_stipsicz_1999} where the framing coefficient on each link component $K$ is equal to $\tb(K)-1$, where $\tb$ denotes the Thurston-Bennequin invariant.
    Given such a diagram, we can easily compute 
    the first Chern class of the Stein domain as follows \cite[Prop. 2.3]{Gompf_HandlebodyConstructionStein}:
    \begin{equation}\label{eq:ChernClassContactStructure}
        c_1(W) = \left[\sum_{i=1}^N \mathrm{rot}(K_i) {h_{K_i}^*}\right] \in H^2(W),
    \end{equation}
    where $W$ is the Stein domain specified by the diagram, $K_1,\dots, K_N$ 
    are the \emph{oriented} Legendrian components in the diagram, $h_{K_i}^*\in C^2(W)$
    denotes the cochain associated to the $2$-handle attached along $K_i$    and $\mathrm{rot}(K_i)$ is the rotation number of the component $K_i$.
    
    We are now ready to define our first  family of manifolds.
    For any $n\in \N$, we define  $X_n$ to be the Stein domain specified by the two component Legendrian link diagram in standard form in Figure~\ref{fig:X_n}.
    
    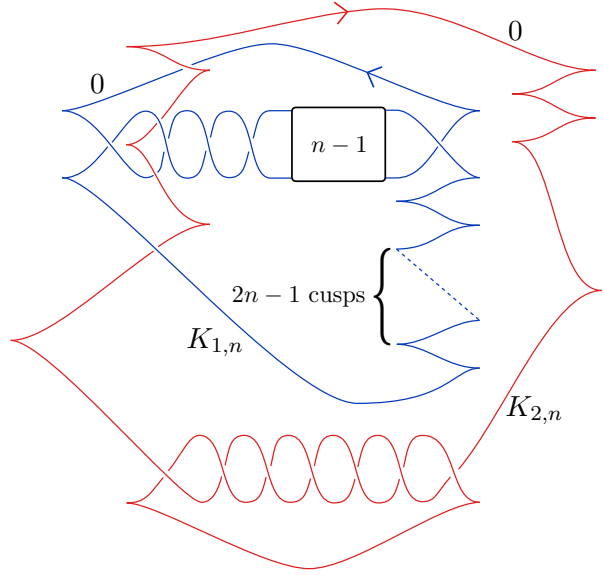
\begin{figure}
        \centering
        \begin{tikzpicture}[scale=0.5]   
            \node at (0,0) {\includegraphics[width=0.5\textwidth]{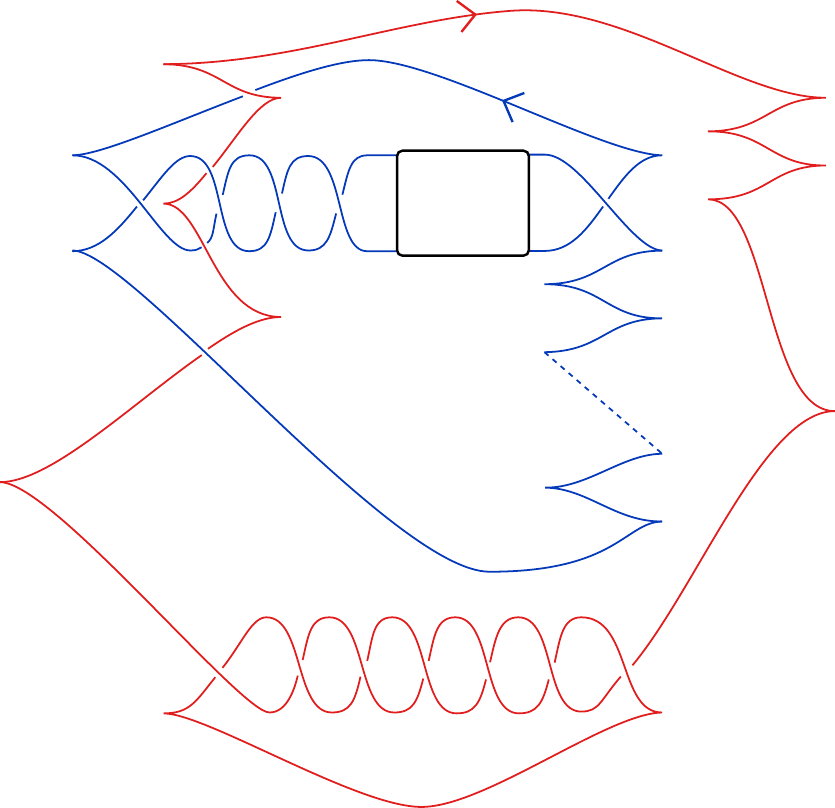}};
            \node at (-2.4,-1.5) {$K_{1,n}$};
            \node at (6,-3.3) {$K_{2,n}$};
            \node[scale=0.83] at (0.87,3.73) {$n-1$};
            \node[scale=0.83] at (-0.2,-0.29) {$2n-1$ cusps};
            \node[scale=1.6,yscale=2.1] at (2,-0.29) {\{};
            \node at (-5.5,5.35) {$0$};
            \node at (5.5,6.75) {$0$};
        \end{tikzpicture}
        \caption{A Legendrian link diagram in standard form for $X_n$. The knot~$K_{1,n}$ with the specified orientation has $2n+3$ crossings, $2n+2$ right cusps and rotation number~$2n$. The knot $K_{2,n}$ has rotation number $0$.}
        \label{fig:X_n}
    \end{figure}

    \begin{thm}\label{thm:FamilyX_n} For each $n\in \N$ the $4$-manifold $X_n$ is simply-connected, has $H_2(X)\cong \Z^2$ and  vanishing intersection form.
    Moreover, all the non-trivial elements of the topological Torelli group
    $ \Tor(X_n,\partial X_n)$ are  non-smoothable.
    Furthermore, define $n_r :=  2^{r+2}-3$ for $r\geq 1$. Then
    $\partial X_{n_r}$ is not diffeomorphic to $\partial X_{n_{m}}$ for any $r\neq m$.
    \end{thm}
    
    \begin{proof} The first part follows from the fact that $X_n$ is a link trace
    on a link with two components which are both $0$-framed and with vanishing linking number. From this it also follows that~$H_1(\partial X_n) \cong \Z^2$. 
    
    We want to invoke  \Cref{cor:existence_condition}.
    By construction $X_n$ is a Stein domain, hence it only remains to show
    that $c_1(\xi_n)\neq 0$, with $\xi_n$ being the contact structure induced on $\partial X_n$.
    
    Denote by $x_{1,n}, x_{2,n} \in H_2( X_n)$ the homology classes induced by the  $2$-handles attached along $K_{1,n}$ and $K_{2,n}$, respectively.
    
    By \eqref{eq:ChernClassContactStructure}, we have that:
    \begin{equation*}
        c_1(X_n) = \mathrm{rot}(K_{1,n})x_{1,n}^* + \mathrm{rot}(K_{2,n})x_{2,n}^* =2n x_{1,n}^*,     
    \end{equation*}
     with respect to the dual basis $x_{1,n}^*, x_{2,n}^*\in H^2(X_n)$.
    The first Chern class of the Stein structure  restricts to that of the contact structure on the boundary and the inclusion map $i_{\partial X_n}\colon\partial X_n\hookrightarrow X_n$ induces an isomorphism $H^2( X_n)\overset{\cong}{\to} H^2(\partial X_n)$, as can be seen from the long exact sequence of the pair and using the fact that the intersection form of $X_n$ is trivial and $H_1(X_n)= 0$.  It follows immediately that \Cref{cor:existence_condition} applies.

    It remains to show that $\partial X_{n_r}\neq \partial X_{n_m}$ for $r\neq m$.
    In the rest of the proof we will identify $H_2(X_n)\cong \Z^2$ via $x_{i,n}\mapsto e_i$, $i=1,2$, and then identify $H_2(\partial X_n) \cong \Z^2$ by composing the previous identification with the isomorphism $H_2(\partial X_n)\overset{\cong}{\to} H_2(X_n)$ given by the inclusion.
    
    For $A \in GL(\Z,2)$, we define 
    \begin{equation*}
        \genus_{X_n}(A) := \max\{g_{X_n}(A \cdot e_1),g_{X_n}(A\cdot e_2)\},
    \end{equation*}
    where $g_{X_n}: H_2(X_n)/\Torsion(H_2(X_n))\to\Z_{\geq 0} $ is the minimal genus function \cite[p. 37]{gompf_stipsicz_1999}.

    The adjunction inequality for Stein domains     \cite{LiscaMatic} (see also \cite[Prop.~9.2]{AkbulutBook} or \cite[Thm.~11.4.7]{gompf_stipsicz_1999}) gives:
    \begin{equation}\label{eq:genusLowerBound}
        \genus_{X_n}(A)\geq 1 + \frac 1 2 \max_{i=1,2}|\langle c_1(X), A \cdot e_i\rangle| 
            \geq 1 + \frac {2n} 2\max \{|A_{1 1}|,| A_{2 1}|\} \geq 1+n
    \end{equation}
    where we used that the intersection form of $X_n$ vanishes and, in the last inequality, we used that  $\det(A) \neq 0$.

    Now suppose for a contradiction that  $f\colon \partial X_{n_r} \to \partial X_{n_{m}}$ 
    is a diffeomorphism.  By swapping to $f^{-1}$ if necessary, we can suppose
    that $r<m$.
    
    Under the identifications introduced above, $f_*\colon H_2(\partial X_{n_r})\to H_2(\partial X_{n_m})$ induces an element~$A \in GL(\Z,2)$.
    
    We can find  closed, orientable surfaces  $\Sigma_{i,n_r}\subset \partial X_{n_r}$   representing $e_i\in H_2(\partial X_{n_r})$, for $i=1,2$   of genera $g(\Sigma_{1,n_r}) = 2n_r +3$ and $g(\Sigma_{2,n_r}) = 7$; see \Cref{fig:Surfaces}. Hence $f(\Sigma_{i,n_r})\subset \partial X_{n_m}$ provides us with the upper bound $\genus_{X_{n_m}}(A) \leq \max \{2n_r +3,7\}\leq 2n_r +3$.

    This together with \eqref{eq:genusLowerBound} gives
    \begin{equation*}
                 2n_r +3  \geq \genus_{X_{n_m}}(A) \geq 1+n_m\geq 1+n_{r+1},
    \end{equation*}
    where in the last inequality we used that $r<m$ and that $r\mapsto n_r$ is increasing with $r$. Now, it can be seen that our definition of $n_r$ satisfies the recursive equation $n_{r+1} = 2 n_{r}+3$, for any $r \in \N$. Thus the  inequality above gives the contradiction 
    \begin{equation*}
        2n_r +3 \geq 1+n_{r+1} = 1+2n_r +3.\qedhere
    \end{equation*}
    \end{proof}

    \begin{figure}
        \centering        
        \includegraphics[scale=0.8]{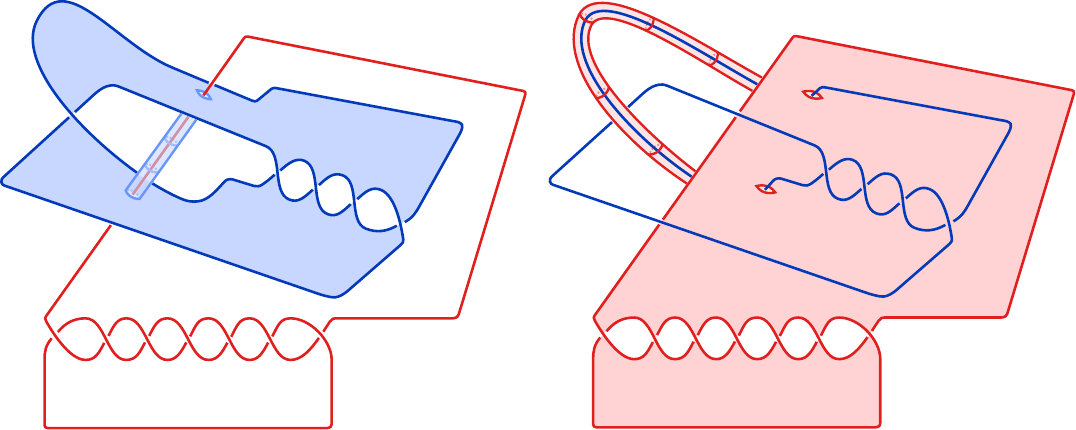}
        \caption{ The pictures show, in a surgery presentation for $\partial X_n$ (every link component is $0$-framed), the two surfaces
         $\Sigma_{1,n}$ (on the left) and  $\Sigma_{2,n}$ (on the right) for the case $n=1$. They have genera $g(\Sigma_{1,n})=2n+3$ and $g(\Sigma_{2,n})=7$.  }
        \label{fig:Surfaces}
    \end{figure}

    \begin{remark} By mimicking the construction of the manifolds $X_n$, it is not difficult to construct another family of manifolds $\{Q_n\}_{n\in \N}$ with $b_1(\partial Q_n)\to \infty$ as $n\to \infty$ and with $\Tor(Q_n, \partial Q_n)$ containing infinitely many non-smoothable mapping classes. 
    For example, one can define $Q_n$ recursively:  start with $Q_1=X_1$ and obtain $Q_{n+1}$ from $Q_n$ by attaching a $2$-handle along a knot $C_{n+1}$,
    where $C_1 = K_{2,1}$ and, for $n\geq 1 $, $C_{n+1}$ is a knot
     identical to $K_{2,1}$, but unlinked from all components besides $C_n$,
     and linking $C_{n}$ in the same way that $K_{2,1}$ links $K_{1,1}$.
Then $Q_n$ will have a Stein structure satisfying the hypothesis of \Cref{cor:existence_condition} and $b_1(\partial Q_n) = 1+n$.
    \end{remark}
        
    \subsection{Family from knot surgery}
    Now we will construct an infinite family of manifolds sharing the same boundary.
    It is possible to give a proof of \Cref{thm:FamilyZ_n} below by pairing \Cref{cor:existence_condition} together with well known compactification results for Stein domains \cite{eliashberg_2004,EtnyreSymplecticFillings, etnyre_min_mukherjee_2022, LiscaMatic} but we decided to use a more elementary approach here which does not rely on symplectic topology.
    
    Let $Z$ be the $4$-manifold with boundary defined by the Kirby diagram in \Cref{fig:Z}~(a).   Let~$T\subset \mathrm{int}(Z)$ be the embedded torus obtained by capping the genus one Seifert surface for the red trefoil knot with the core of the handle attached along it. From the diagram it is clear  that $[T]\neq 0 \in H_2(Z)$ and $[T]^2 = 0$.
    
    For any $n\in \N$ we define the knot $K(n)$ to be the twist knot with Alexander polynomial $\Delta_{K(n)} = -(2n-1) + n(t+t^{-1})$, and $E_{K(n)}$ to be the knot exterior of $K(n)$ in $S^3$.  Then we define $Z_n:=E(K)\times S^1\cup_\partial (Z\sm \nu T)$ to be the manifold obtained by performing knot surgery \cite{FintushelSternKnotSurgery}
    on the torus $T$ using the knot $K(n)$.  Since the knot surgery only changes the manifold in the interior, we have an identification $\partial Z_n \cong Y:= \partial Z$.

\begin{figure}
    \centering
    \begin{tikzpicture}[scale=1]   
        \node at (0,0) {\includegraphics[width=\textwidth]{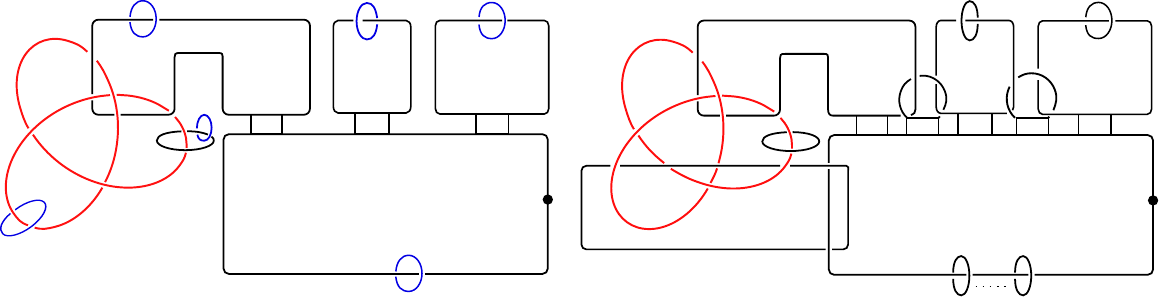}};
        \node at (-4.5,-2.3) {\textbf{(a)}};
        \node at (-7.4,1.6) {$0$};
        \node at (-4.5,1.95) {$0$};
        \node at (-5.11,-0.23) {$-2$};
        \node at (-2.3,1.95) {$0$};
        \node at (-0.5,1.95) {$0$};
        \node at (-7.45,-1.4) {$\mu_1$};
        \node at (-5.84,2.15) {$\mu_2$};
        \node at (-1.15,2.15) {$\mu_3$};
        \node at (-5.05,0.6) {$\mu_4$};
        \node at (-2.8,2.15) {$\mu_5$};
        \node at (-2.23,-2.1) {$\mu_6$};
        \node at (-4.2,0.31) {$\scriptstyle 2$};
        \node at (-2.78,0.33) {$\scriptstyle 3$};
        \node at (-1.16,0.32) {$\scriptstyle 3$};
        \node at (3.5,-2.3) {\textbf{(b)}};
        \node at (0.4,1.4) {$0$};
        \node at (3.5,1.95) {$0$};
        \node[scale=0.9] at (3.01,0.35) {$-2$};
        \node at (7,2.15) {$-1$};
        \node at (5.8,1.95) {$0$};
        \node at (4.2,0.95) {$1$};
        \node at (6.52,0.95) {$1$};
        \node at (5.22,2.15) {$-1$};
        \node at (4.72,-1.95) {$-1$};
        \node at (6.28, -1.95) {$-1$};
        \node at (1.5,-1.55) {$-1$};
        \node at (7.5,1.95) {$0$};
        \node at (3.93,0.31) {$\scriptstyle 2$};
        \node[scale=0.8] at (4.61,0.295) {$\scriptstyle 3$};
        \node at (5.32,0.31) {$\scriptstyle 3$};
        \node[scale=0.8] at (6.10,0.295) {$\scriptstyle 3$};
        \node at (6.94,0.31) {$\scriptstyle 3$};
        \node[rotate=90, scale=2] at (5.55,-2.18) {$\{$};
        \node at (5.55,-2.49) {$15$ copies};
    \end{tikzpicture}
    \caption{(a) Kirby diagram for the $4$-manifold $Z$, (b) Kirby diagram   showing an embedding of $Z$ into $K_3\# 2\ol{\CP}^2$.}
    \label{fig:Z}
\end{figure}
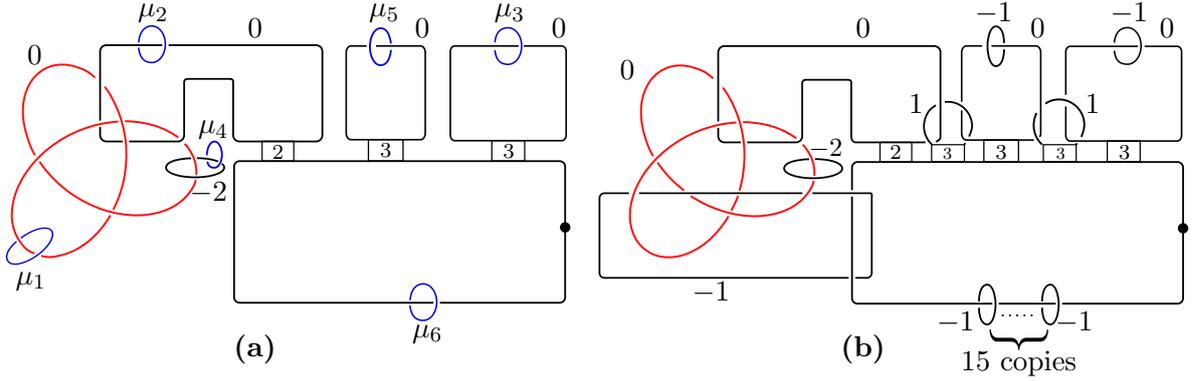

    \begin{prop}\label{prop:boyer}
        The $4$-manifolds $\{Z_n\}_{n\in\N}$ are all homeomorphic to $Z$ relative to $Y$, moreover they  are all simply-connected with intersection form $ (\Z^2\oplus \Z^2,\left[\begin{smallmatrix}
                 0 & 1 \\ 
                 1 & -2 \\
                \end{smallmatrix}\right]\oplus 0).$
    \end{prop}

    \begin{proof}
        In \cite{FintushelSternKnotSurgery} it is shown that knot surgery preserves the homeomorphism type of a closed, simply-connected manifold provided that the surgered torus has simply-connected complement, but an additional argument is required when the manifold has non-empty boundary.  The torus $T\subset Z$ embeds in a Gompf nucleus $N(2)$ \cite{GompfNuclei} by construction  and  so, in particular, it has simply-connected complement. This implies that all of the $Z_n$ are simply-connected by a routine Seifert-Van Kampen argument.  Boyer \cite[Thm.~0.7]{boyer_1986} tells us that there are two obstructions to extending the homeomorphism $\Id\colon Y\to Y$ to a homeomorphism $f\colon Z_n\to Z$.  The first is to find a `morphism', namely an isometry $\Lambda$ of the intersection forms such that
    \[
    \begin{tikzcd}
        H_2(Y) \arrow[equal]{d} \arrow[r] &  H_2(Z_n) \arrow[r] \arrow[d, "\Lambda"] &  H_2(Z_n)^* \arrow[r] & H_1(Y) \arrow[equal]{d} \\
        H_2(Y) \arrow[r] & H_2(Z) \arrow[r] & H_2(Z)^* \arrow[u,"\Lambda^*"] \arrow[r] & H_1(Y)
    \end{tikzcd}
    \]
    commutes, where the rows in the above diagram are the same rows used in \Cref{rem:minimality}. 
    
    We can find such a $\Lambda$ in the following way.  Define a map $g\colon Z_n\to Z$ as the identity on the complement of $T$ and as the standard degree one map from $E_{K(n)}\to \mathbb D^2\times \mathbb S^1$ (see \cite[Prop. 1]{boileau_boyer_rolfsen_wang}) times the identity map on the $\SS^1$-factor on $E_{K(n)}\times \SS^1$ (for the standard degree one map m.  
    Then $\Lambda:=g_*$ is an isometry of the intersection forms.  The fact that $\Lambda$ is a morphism in the sense of Boyer follows by replacing the third column in the above diagram by the relative homology groups and noting that this new diagram  commutes by naturality.  Note that the existence of an isometry of the intersection forms implies that all of the manifolds $Z_n$ are spin since they have even intersection forms ($Z$ having the stated intersection form can be verified by looking at the Kirby diagram in \Cref{fig:Z}).  Boyer's second obstruction vanishes if we can show that the unique spin structure on $Z_n$ and the unique spin structure on $Z$ restrict to the same spin structure on $Y$.  To show this, we will now be more explicit about the gluing maps used in the knot surgery.  
    
    We have an identification $\partial(E(K)\times \SS^1)\cong \SS^1\times \SS^1\times \SS^1$ where a longitude to $K$ is identified with the third $\SS^1$-factor, a meridian to $K$ is identified with the first $\SS^1$-factor, and the remaining $\SS^1$-factor is identified with the second $\SS^1$-factor.  Similarly, we have an identification $\partial (Z\sm \nu T)\cong T\times \SS^1\cong \SS^1\times \SS^1 \times \SS^1$ in the obvious way.  This gives an identification $\partial (E(K)\times \SS^1)\cong \partial (Z\sm\nu T)$ and we use this to perform the knot surgery.  Consider the unique spin structure $\sigma$ on $Z$ restricted to $\partial (Z\sm\nu T)\cong \SS^1\times \SS^1\times \SS^1$.  Note that $\sigma$ extends over $\nu T\cong T\times \mathbb D^2$, and hence there are four possibilities for $\sigma\vert_{\SS^1\times \SS^1\times \SS^1}$.  Conversely, there are four choices of spin structures on $E(K)\times \SS^1$, and, by restricting to the boundary, these give rise to four distinct spin structures on $\partial (E(K)\times \SS^1)\cong \SS^1\times \SS^1\times \SS^1$.  Using the degree one map $E(K)\times \SS^1\to T\times  \mathbb{D}^2$, we see that  these are precisely the four spin structures which extend over $T\times \mathbb D^2$, and so regardless of how $\sigma$ restricts to $\partial(Z\sm\nu T)$, we can pick a spin structure on $E(K)\times \SS^1$ such that $\sigma\vert_{Z\sm\nu T}$ extends to a spin structure on $Z_n$.  By construction, these two spin structures clearly match on $Y$, and hence we have a homeomorphism $Z\to Z_n$ which restricts to the identity map on $Y$.
    \end{proof}

    \begin{thm}\label{thm:FamilyZ_n} The $4$-manifolds $\{Z_n\}_{n\in \N}$ are all homeomorphic relative to $Y$, but pairwise not diffeomorphic relative to $Y$. They are all simply-connected with intersection form $ (\Z^2\oplus \Z^2,\left[\begin{smallmatrix}
                 0 & 1 \\ 
                 1 & -2 \\
                \end{smallmatrix}\right]\oplus 0)$
    and have infinite $\Tor(Z_n,Y)$.
    Moreover, all  non-trivial 
    elements of the Torelli group $\Tor(Z_n, Y)$ are  non-smoothable.
    \end{thm}

    \begin{proof} The first part of the first statement follows directly from \Cref{prop:boyer}.  In \Cref{fig:Z}(b) we depict an embedding of $Z$ into $X := K_3\#2\ol{\CP}^2$, whose Kirby diagram has been taken from \cite[Fig. 8.16]{gompf_stipsicz_1999} (see also \cite{AKMRStableIsotopy}).
    Hence $Z_n$ embeds into the closed manifold $X_n$ obtained by performing knot surgery on $T\hookrightarrow X$ using $K(n)$.
    From \cite{FintushelSternKnotSurgery} and the blow-up formula \cite{fintushel_stern_blowup}, it follows that the manifolds $X_n$ are pairwise non-diffeomorphic.  Indeed, the Seiberg-Witten invariant of $X_n$, seen as an element of the group ring $\Z[H^2(X_n)]$, is equal to 
    \begin{equation}\label{eq:SWOfX_n}
        SW(X_n) = (E_1+E_1^{-1})(E_2+E_2^{-1})\left(-(2n-1) + n(F^2+F^{-2})\right),
    \end{equation}
    where $E_i\in H^2(X_n)$ are the classes coming from the two blow-ups and $F:=\PD[T]$ is the Poincar\'e dual to the torus $T$.
    Now, since $X_n$ is obtained by capping $Z_n$ with a fixed manifold $Q:= X\sm Z$ independent from $n$,  the manifolds $Z_n$ are pairwise non-diffeomorphic relative to their boundaries.
    
    It remains to prove the last statement of the theorem.  We want to apply \Cref{Lemma:KeyLemma}, so we check that the hypotheses hold.  We have that $H_1(Y)$ is isomorphic to $\Z^2$ generated by $v_1:= \mu_2+\mu_3$ and $v_2 := \mu_5$, where the $\mu_i$
    are the meridians to the components as shown in \Cref{fig:Z}(a). In particular from Theorem~\autoref{thm:orson_powell_variations} this implies that $\Tor(Z_n,Y)\cong \Z$ and hence is infinite.
    From \eqref{eq:SWOfX_n} we see that $E_1+E_2\in H^2(X_n)$ ($E_1E_2$ in group ring notation) is a Seiberg-Witten basic class for~$X_n$, which satisfies $\PD i^*_{Y,X}(E_1+E_2) = \mu_3+\mu_5$, because $\PD (E_1)$ and $\PD(E_2)$ are represented by the cores of the 2-handles which in Figure~\ref{fig:Z}(b) are attached along $-1$-framed curves corresponding to $\mu_3$ and $\mu_5$ union the obvious discs in $\mathrm{int}(Z_n)$ that those curves bound. Hence the intersection of these surfaces with $Y$ yields $\mu_3$ and $\mu_5$.

    The dotted circle in Figure~\ref{fig:Z}(a) gives the relation $2\mu_2+3\mu_3 + 3\mu_5 = 0$ in $H_1(Y)$  \cite[Prop. 5.3.11]{gompf_stipsicz_1999}, thus $\mu_3+\mu_5 = -2(v_1+v_2)\neq 0$ and  hence is non-torsion. 
    Moreover, the complement $Q$ is
    obtained by adding only $2$-handles and a single $4$-handle to $Y$ and
    hence $H_3(X_n,Z_n)\cong H_3(Q,Y) \cong H^1(Q)=0$. Now the final statement of the theorem follows from \Cref{Lemma:KeyLemma}.
    \end{proof}
   
    \section{Generalised Dehn twists}\label{sec:gen_dehn_twists}
\subsection{Absolute non-smoothability and generalised Dehn twists}
We begin by reviewing absolute and relative smoothability from the
point of view of spaces of maps. Given a smooth, compact, oriented manifold $X$ with boundary  (not necessarily of dimension four) we denote by $\Diff^+(X)$  the set of orientation-preserving self-diffeomorphisms of $X$ topologised with the $C^\infty$-topology, and by $\Diff^+(X,\partial X)\subset \Diff^+(X)$ the subspace of diffeomorphisms restricting to the identity over $\partial X$. 
Similarly we define $\Homeo^+(X)$ and $\Homeo^+(X,\partial X)$ using the compact-open topology. 
The inclusion map 
\begin{equation*}
    (\Diff^+(X), \Diff(X,\partial X)) \to (\Homeo^+(X), \Homeo^+(X,\partial X))
\end{equation*}
 is continuous.  We will denote by 
\begin{equation*}
\begin{split}
     &\Phi:\pi_0 \Diff^+(X)\to \pi_0\Homeo^+(X)\\
   & \Phi_\partial: \pi_0 \Diff^+(X,\partial X)\to \pi_0\Homeo^+(X,\partial X)\\
\end{split}
\end{equation*}
the induced maps on the mapping class groups.

With these definitions in place, (relative) non-smoothability can be given an alternative definition by saying that $\varphi \in \pi_0\Homeo^+(X,\partial X)$ is non-smoothable if $\varphi \not \in \Image{\Phi_\partial}$.

\begin{definition}
    Let $i\colon \pi_0\Homeo^+(X,\partial X)\to \pi_0\Homeo^+(X)$ be the map induced by the inclusion.  We say that $\varphi \in \pi_0\Homeo^+(X,\partial X)$ is \emph{absolutely} non-smoothable if $i(\varphi) \in \pi_0\Homeo^+(X)$ does not belong to $\Image{\Phi}$.
\end{definition}
Explicitly, the difference between a relatively and an absolutely smoothable homeomorphism is that in the latter case the isotopy at each time does not need to fix the boundary pointwise.
Surprisingly, for $4$-dimensional manifolds the two notions of absolute and relative non-smoothability coincide. An important role in the  proof 
is played by generalised Dehn twists \cite[Sec.~1.2]{orson_powell_2023}, which we now review.

\begin{definition}
    Let $X$ be a compact, smooth, oriented $n$-manifold with boundary. Given  $[\gamma]\in\pi_1 \Diff^+(\partial X)$, we define the \emph{generalised Dehn twist} with respect to $[\gamma]$ to be the smooth isotopy class of the diffeomorphism $\varphi_\gamma\colon X\to X$ defined on a collar of $\partial X$ as  $\varphi(y,t)=(\gamma(t)(y),t)\in (\partial X)\times I$ and defined outside of the collar as the identity map. 
\end{definition}

Another point of view is the following. 
The sequence of inclusion and restriction 
\begin{equation*}
    \Diff^+(X,\partial X){\to} \Diff^+(X) \to \Diff^+(\partial X)
\end{equation*}
and the equivalent sequence in the topological category are fibration sequences \cite{lashof_1976}.  It can be shown that the connecting morphism $\pi_1\Diff^+(\partial X)\to \pi_0\Diff^+(X,\partial X)$ of the long exact sequence of homotopy groups
is precisely the map that associates to a loop of diffeomorphisms $[\gamma]$ its generalised Dehn twist $[\varphi_\gamma]$ \cite[Sec.~1.4]{orson_powell_2023}.

\begin{lem} \label{lem:absolutelyNonSmoothable} 
Let $X$ be a compact, smooth, oriented,  $4$-manifold with boundary.  Then a mapping class $\varphi \in \pi_0 \Homeo^+(X, \partial X)$ is (relatively) non-smoothable if and only if it is absolutely non-smoothable. 
\end{lem}
\begin{proof}
    We will prove that relative non-smoothability implies absolute non-smoothability, the other implication is clear. 
    
     We have the following  commutative diagram with exact rows:
     
    \begin{tikzcd}
       \pi_1 \Diff^+(\partial X) \arrow[r]\arrow[equal]{d} & \pi_0 \Diff^+(X,\partial X) \arrow[r, "i"]\arrow[d,"\Phi_\partial"] & \pi_0 \Diff^+(X) \arrow[r, "\partial"] \arrow[d, "\Phi"]& \pi_0 \Diff^+(\partial X) \arrow[equal]{d} \\
        \pi_1 \Diff^+(\partial X) \arrow[r] & \pi_0 \Homeo^+(X,\partial X) \arrow[r,"i"] & \pi_0 \Homeo^+(X) \arrow[r, "\partial" ]& \pi_0 \Diff^+(\partial X)  
    \end{tikzcd}
    \newline
    where $i$ denotes  the maps induced by the inclusion and we are implicitly using the well known homotopy equivalence $\Homeo(Y)\simeq \Diff(Y)$ for  any $3$-manifold $Y$  \cite{cerf_1959, HatcherSmaleConj}. Note that this is where we need the assumption that $\dim X = 4$. 
    
    Let $\phi \in \pi_0 \Homeo^+(X,\partial X)$ be non-smoothable.
    Suppose for a contradiction that there exists $\psi \in \pi_0\Diff^+(X)$ such that
    $\Phi(\psi) = i(\phi)$.
    Then since $\partial (i(\varphi)) = \Id_{\partial X}$,  the commutativity and exactness of the diagram implies that there exists $\psi'\in \pi_0\Diff^+(X,\partial X)$ 
    such that $i(\psi') = \psi$. 
    Hence $\Phi_\partial(\psi') $ is equal to  $ \phi$  modulo composition
    with an element in the image of $\pi_1 \Diff^+(\partial X) \to \pi_0 \Homeo^+(X,\partial X)$. Thus $\Phi_\partial(\psi') = [\varphi_\gamma] \circ \phi  $ for some generalised Dehn twist $[\varphi_\gamma]$, but then 
    $\phi = [\varphi_\gamma]^{-1}\circ\Phi_\partial(\psi')$ presents $\phi$ as a composition of diffeomorphisms, contradicting the non-smoothability of $\phi$.\end{proof}
    \subsection{Realising smoothable elements of the Torelli group by generalised Dehn twists}
Since generalised Dehn twists are supported in a collar of the boundary, it is clear that these give rise to smooth elements in the Torelli group of the $4$-manifold.  One could ask whether generalised Dehn twists generate the whole  Torelli group. 
The next proposition gives an answer under the assumption that the boundary is connected and prime; the general case is still unknown to the authors' best knowledge. 

\begin{prop} \label{Prop:TorDehnrealised} Let $X$ be a smooth, compact, simply-connected, oriented $4$-manifold with connected and prime boundary $Y$. Then the topological Torelli group $\Tor(X,Y)$ is realised by generalised Dehn twists  if and only if one of the following holds:
\begin{enumerate}
    \item $b_1(Y)<2$,
    \item $b_2(Y) = 2$ and $Y$ is Seifert fibered with base orbifold  $\mathbb T^2$,
    \item $Y = \mathbb T^3$,
\end{enumerate}
where $\mathbb{T}^n$ denotes the $n$-torus.
\end{prop}
\begin{proof}
    We begin by showing that (1) or (3) implies that $\Tor(X,Y)$ is realised by generalised Dehn twists.  First suppose that $b_1(Y)<2$. Then $\Lambda^2H_1(Y)^* = 0$ and hence
    $\Tor(X,Y) $ is trivial. The case $Y= \mathbb T^3$ can be handled by applying  \cite[Prop. 8.9]{orson_powell_2023} to the three  generalised Dehn twists induced by the three $\SS^1$-factors. More precisely, let $\alpha_1,\alpha_2,\alpha_3$ be the basis of $H_1(Y)$ induced by the $\SS^1$-factors of $\mathbb T^3=\SS^1\times \SS^1\times \SS^1$ and let $\alpha_1^*, \alpha_2^*,\alpha_3^*\in H_1(Y)^*$ be the dual basis.  Then the element of $\Lambda^2 H_1(Y)^*$ associated
    to a rotation of the $i$-th $\SS^1$-factor is  $\pm \alpha_k^*\wedge \alpha_j^*$, where $k, j\neq i$ \cite[Prop.8.9]{orson_powell_2023}, and hence the three rotations generate the whole Torelli group.

    We now show that if neither (1) nor (3) hold, then either (2) holds or $\Tor(X,Y)$ is not realised by generalised Dehn twists. 
    So assume that $Y\neq \mathbb T^3$ and $b_1(Y)\geq 2$.
    In this case $Y$ is Haken \cite[1.1.6]{Waldhausen68} (therein called \emph{sufficiently large}),  and  \cite{HatcherSufficientlyLarge,laudenbach} implies that $\pi_1 \Diff(Y)\cong \pi_1(\mathrm{hAut}(Y))$.  Then it is a result of Gottlieb \cite[Thm.III.2]{gottlieb} that $\pi_1(\mathrm{hAut}(Y))\cong Z(\pi_1(Y))$, hence in particular is abelian.
    Then it follows from \cite[Satz 4.1]{WaldhausenZentrum} that either the center $Z(\pi_1(Y))$ is trivial or $Y$ is Seifert fibered over an orientable orbifold.
    In the former case, $b_1(Y)\geq 2$ implies that the Torelli group, being non-trivial, cannot be generated by generalised Dehn twists.
    In the latter case, $Z(\pi_1(Y))\cong \Z$ generated by a principal orbit of the $\SS^1$-action \cite{WaldhausenZentrum}, hence $\pi_1 \Diff(Y)\to \Tor(X,Y)$ cannot be surjective if $b_1(Y) > 2$, for in this case $\Tor(X,Y)$ has rank at least two.

    We finish by showing that (2) implies that $\Tor(X,Y)$ is realised by generalised Dehn twists.  When $b_1(Y)=2$ and $Y$ is Seifert fibered over an orientable orbifold, the quotient is necessarily $\mathbb T^2$ \cite{BrydenLawsonPigottZvengrowksi}. Moreover  the variation associated to the $\SS^1$-action is computed in \cite[Prop. 8.9]{orson_powell_2023} and 
    in this case it generates the whole of $\Lambda^2 H_1(Y)^*\cong \Z$.\end{proof}
    
    In particular, if the boundary satisfies any of the three conditions of \Cref{Prop:TorDehnrealised} then it is impossible to find a non-smoothable homeomorphism in the Torelli group.  
    
    Given the existence of non-smoothable elements of the Torelli group, we can say more.  It is possible to find smoothable elements of the Torelli group  which are not isotopic to any diffeomorphism supported on a collar of the boundary, let alone are realised by generalised Dehn twists.

    To state the next theorem, recall that given two homeomorphisms of connected $4$-manifolds $f\colon X_1\to X_1$ and $g\colon X_2\to X_2$, we can form the connect-sum homeomorphism $f\# g$ by first performing isotopies of $f$ and $g$ such that they restrict to the identity map on the discs used to perform the connect-sum.  This fact follows from isotopy extension \cite{edwards_kirby_1971}, uniqueness of normal bundles \cite[Chapter 9.3]{freedman_quinn_1990}, and the fact that any orientation-preserving diffeomorphism $f\colon S^3\to S^3$ is isotopic to the identity \cite{cerf_s3}.
    
\begin{thm}\label{thm:gen_dehn_twist}
    Let $X$ be a smooth, simply-connected, oriented,   compact $4$-manifold with boundary such that there exists a non-smoothable self-homeomorphism $\varphi\in \Tor(X,\partial X)$.  Then there exists an integer $m\geq 1$ such that  \[
    \varphi\# \Id\colon X\#m(\SS^2\times \SS^2)\to X\#m(\SS^2\times \SS^2)
    \]
    is a smoothable homeomorphism not isotopic to any smooth map supported on a collar of the boundary.
\end{thm}

The proof relies on the following result.

\begin{lem}\label{lem:cs}
    Let $X$ be a smooth, simply-connected, compact, oriented $4$-manifold with boundary and $\varphi\colon X\to X$ a self-homeomorphism.  Then there exists an integer $m\geq 1$ such that \[\psi:=\varphi\#\Id\colon X\#m(\SS^2\times \SS^2)\to X\#m(\SS^2\times \SS^2)
    \]
    is isotopic to a diffeomorphism relative to the boundary.
\end{lem}

This result is proved in \cite[Sec.~8.6]{freedman_quinn_1990}.  For more details on the proof, see \cite[Proposition 2.23]{galvin_2024}.

\begin{proof}[Proof of \Cref{thm:gen_dehn_twist}]
    Let $\varphi\in \Tor(X,\partial X)$ be one of the non-smoothable mapping classes.  By \Cref{lem:cs} there exists an integer $m\geq 1$ such that \[
    \psi:=\varphi\#\Id\colon X\#m(\SS^2\times \SS^2)\to X\#m(\SS^2\times \SS^2)
    \]
    is isotopic to a diffeomorphism.  Since $\psi$ was defined by extending $\varphi$ via the identity onto the $\SS^2\times \SS^2$ summands, we also have that $\psi\in \Tor(X\#m(\SS^2\times \SS^2))$.  
    Now assume for a contradiction that   $\psi$ is supported on a collar of $\partial X\#m(\SS^2\times \SS^2)\cong \partial X$. Then we can remove the $\SS^2\times \SS^2$ summands and obtain a diffeomorphism  $\psi'\colon X\to X$.  However, since we have an identification \[
    \Tor(X,\partial X)\cong \Tor(X\#m(\SS^2\times \SS^2))\]
    we can see that $\Delta_{\psi'}=\Delta_{\varphi}$.  Hence, by \Cref{thm:orson_powell_variations} we see that $\psi'$ and $\varphi$ must be isotopic relative to the boundary.  This contradicts the assumption that $\varphi$ was not isotopic to a diffeomorphism, and so we conclude that $\psi$ is not isotopic to any smooth map supported on a collar of the boundary.
\end{proof}
    
As an immediate corollary, we have that there exist examples of $4$-manifolds with boundary where all elements of the Torelli group are smoothable, but all non-trivial elements are not supported on a collar of the boundary. This was \Cref{thm:2} from the introduction.

\begin{cor}
    There exists an infinite family of smooth, compact, oriented, simply-connected $4$-manifolds with connected boundary $(W_n,\partial W_n)$ and $\Tor(W_n,\partial W_n)$ infinite order such that all mapping classes in $\Tor(W_n,\partial W_n)$ are smoothable, but only the identity map is supported on a collar of the boundary and, in particular, only the identity map is realised by a generalised Dehn twist.
\end{cor}

\begin{proof}
    Apply \Cref{thm:gen_dehn_twist} to the family $X_n$ from \Cref{thm:FamilyX_n}.
\end{proof}

\bibliographystyle{alpha}
\bibliography{bibliography.bib}

\newcommand{\etalchar}[1]{$^{#1}$}
\begin{thebibliography}{AKMR15}

\bibitem[Akb16]{AkbulutBook}
Selman Akbulut.
\newblock {\em 4-manifolds}, volume~25.
\newblock Oxford University Press, 2016.

\bibitem[AKMR15]{AKMRStableIsotopy}
Dave Auckly, Hee~Jung Kim, Paul Melvin, and Daniel Ruberman.
\newblock {Stable isotopy in four dimensions}.
\newblock {\em J. London Math. Soc.}, 91(2):439--463, 2015.

\bibitem[Bar21]{baraglia_2021}
David Baraglia.
\newblock Constraints on families of smooth 4-manifolds from {Bauer}-{Furuta}
  invariants.
\newblock {\em Algebr. Geom. Topol.}, 21(1):317--349, 2021.

\bibitem[BBRW16]{boileau_boyer_rolfsen_wang}
M.~Boileau, S.~Boyer, D.~Rolfsen, and S.~C. Wang.
\newblock One-domination of knots.
\newblock {\em Ill. J. Math.}, 60(1):117--139, 2016.

\bibitem[BF04]{BauerFuruta}
Stefan Bauer and Mikio Furuta.
\newblock A stable cohomotopy refinement of {S}eiberg-{W}itten invariants: I.
\newblock {\em Invent. Math.}, 155(1):1--19, 2004.

\bibitem[BLPZ03]{BrydenLawsonPigottZvengrowksi}
J.~Bryden, T.~Lawson, B.~Pigott, and P.~Zvengrowski.
\newblock The integral homology of orientable {S}eifert manifolds.
\newblock {\em Topol. Appl.}, 127(1):259--275, 2003.
\newblock Proceedings of the Pacific Institute for the Mathematical Sciences
  workshop ``Invariants of Three-Manifolds".

\bibitem[Boy86]{boyer_1986}
Steven Boyer.
\newblock Simply-connected 4-manifolds with a given boundary.
\newblock {\em Trans. Am. Math. Soc.}, 298:331--357, 1986.

\bibitem[Cer59]{cerf_1959}
Jean Cerf.
\newblock Groupes d'automorphismes et groupes de diff{\'e}omorphismes des
  vari{\'e}t{\'e}s compactes de dimension 3.
\newblock {\em Bull. Soc. Math. Fr.}, 87:319--329, 1959.

\bibitem[Cer68]{cerf_s3}
Jean Cerf.
\newblock {\em Sur les diff{\'e}omorphismes de la sph{\`e}re de dimension trois
  {{\(\Gamma_ 4 = 0\)}}}, volume~53 of {\em Lect. Notes Math.}
\newblock Springer, Cham, 1968.

\bibitem[Don90]{donaldson_1990}
S.~K. Donaldson.
\newblock Polynomial invariants for smooth four-manifolds.
\newblock {\em Topology}, 29(3):257--315, 1990.

\bibitem[EK71]{edwards_kirby_1971}
R.~D. Edwards and R.~C. Kirby.
\newblock Deformations of spaces of imbeddings.
\newblock {\em Ann. Math. (2)}, 93:63--88, 1971.

\bibitem[Eli04]{eliashberg_2004}
Yakov Eliashberg.
\newblock {A few remarks about symplectic filling}.
\newblock {\em Geom. Topol.}, 8(1):277 -- 293, 2004.

\bibitem[EMM22]{etnyre_min_mukherjee_2022}
{John B.} Etnyre, Hyunki Min, and Anubhav Mukherjee.
\newblock On 3-manifolds that are boundaries of exotic 4-manifolds.
\newblock {\em Trans. Am. Math. Soc.}, 375(6):4307--4332, June 2022.

\bibitem[Etn04]{EtnyreSymplecticFillings}
John~B. Etnyre.
\newblock On symplectic fillings.
\newblock {\em Algebr. Geom. Topol.}, 4:73--80, 2004.

\bibitem[FM88]{friedman_morgan_1988}
Robert Friedman and John~W. Morgan.
\newblock On the diffeomorphism types of certain algebraic surfaces. {I}.
\newblock {\em J. Differ. Geom.}, 27(2):297--369, 1988.

\bibitem[FQ90]{freedman_quinn_1990}
Michael~H. Freedman and Frank~S. Quinn.
\newblock {\em Topology of 4-manifolds}, volume~39 of {\em Princeton Math.
  Ser.}
\newblock Princeton, NJ: Princeton University Press, 1990.

\bibitem[Fre82]{freedman_1982}
Michael~H. Freedman.
\newblock The topology of four-dimensional manifolds.
\newblock {\em J. Differ. Geom.}, 17:357--453, 1982.

\bibitem[FS95]{fintushel_stern_blowup}
Ronald Fintushel and Ronald~J. Stern.
\newblock Immersed spheres in 4-manifolds and the immersed {Thom} conjecture.
\newblock {\em Turk. J. Math.}, 19(2):145--157, 1995.

\bibitem[FS98]{FintushelSternKnotSurgery}
Ronald Fintushel and Ronald~J. Stern.
\newblock Knots, links, and {$4$}-manifolds.
\newblock {\em Invent. Math.}, 134(2):363--400, 1998.

\bibitem[Gal24]{galvin_2024}
Daniel~A.P. Galvin.
\newblock The {Casson}-{Sullivan} invariant for homeomorphisms of
  {$4$}-manifolds.
\newblock Preprint available at https://arxiv.org/pdf/2405.07928, 2024.

\bibitem[Gei08]{geiges_book}
Hansjörg Geiges.
\newblock {\em An Introduction to Contact Topology}.
\newblock Cambridge Studies in Advanced Mathematics. Cambridge University
  Press, 2008.

\bibitem[GGH{\etalchar{+}}23]{gabai2023pseudoisotopiessimplyconnected4manifolds}
David Gabai, David~T. Gay, Daniel Hartman, Vyacheslav Krushkal, and Mark
  Powell.
\newblock Pseudo-isotopies of simply connected 4-manifolds.
\newblock Preprint available at https://arxiv.org/abs/2311.11196.pdf, 2023.

\bibitem[Gom91]{GompfNuclei}
Robert~E. Gompf.
\newblock Nuclei of elliptic surfaces.
\newblock {\em Topology}, 30(3):479--511, 1991.

\bibitem[Gom98]{Gompf_HandlebodyConstructionStein}
Robert~E. Gompf.
\newblock Handlebody construction of {S}tein surfaces.
\newblock {\em Ann. Math.}, 148(2):619--693, 1998.

\bibitem[Got65]{gottlieb}
Daniel~H. Gottlieb.
\newblock A certain subgroup of the fundamental group.
\newblock {\em Am. J. Math.}, 87:840--856, 1965.

\bibitem[GS99]{gompf_stipsicz_1999}
Robert~E. Gompf and Andr{\'a}s~I. Stipsicz.
\newblock {\em 4-manifolds and {Kirby} calculus}, volume~20 of {\em Grad. Stud.
  Math.}
\newblock Providence, RI: American Mathematical Society, 1999.

\bibitem[Hat76]{HatcherSufficientlyLarge}
Allen~E. Hatcher.
\newblock Homeomorphisms of sufficiently large {$P^2$}-irreducible 3-manifolds.
\newblock {\em Topology}, 15(4):343--347, 1976.

\bibitem[Hat83]{HatcherSmaleConj}
Allen~E. Hatcher.
\newblock A proof of the {S}male conjecture, {$\operatorname{Diff}(S^3) \simeq
  O(4)$}.
\newblock {\em Ann. Math.}, 117(3):553--607, 1983.

\bibitem[KLS18]{KhandhawitLinSasahira}
Tirasan Khandhawit, Jianfeng Lin, and Hirofumi Sasahira.
\newblock {Unfolded Seiberg–Witten Floer spectra, I: Definition and
  invariance}.
\newblock {\em Geom. Topol.}, 22(4):2027 -- 2114, 2018.

\bibitem[KM07]{KM}
P.B. Kronheimer and T.S. Mrowka.
\newblock {\em Monopoles and Three-Manifolds}.
\newblock New Mathematical Monographs. Cambridge University Press, 2007.

\bibitem[KT22]{konno_taniguchi_2022}
Hokuto Konno and Masaki Taniguchi.
\newblock The groups of diffeomorphisms and homeomorphisms of 4-manifolds with
  boundary.
\newblock {\em Adv. Math.}, 409 A:58, 2022.
\newblock Id/No 108627.

\bibitem[Las76]{lashof_1976}
R.~Lashof.
\newblock Embedding spaces.
\newblock {\em Ill. J. Math.}, 20:144--154, 1976.

\bibitem[Lau74]{laudenbach}
Francois Laudenbach.
\newblock {\em Topologie de la dimension trois: homotopie et isotopie},
  volume~12 of {\em Ast{\'e}risque}.
\newblock Soci{\'e}t{\'e} Math{\'e}matique de France (SMF), Paris, 1974.

\bibitem[LM97]{LiscaMatic}
P.~Lisca and G.~Mati{\'c}.
\newblock Tight contact structures and {Seiberg}--{Witten} invariants.
\newblock {\em Invent. Math.}, 129(3):509--525, 1997.

\bibitem[Man03]{ManolescuSWFloerHomotopy}
Ciprian Manolescu.
\newblock {Seiberg-Witten-Floer stable homotopy type of three-manifolds with
  $b_1=0$}.
\newblock {\em Geom. Topol.}, 7(2):889--932, 2003.

\bibitem[MMP20]{ManolescuMarengonPiccirillo}
Ciprian Manolescu, Marco Marengon, and Lisa Piccirillo.
\newblock Relative genus bounds in indefinite four-manifolds.
\newblock 2020.

\bibitem[MS97]{morgan_szabo_1997}
John~W. Morgan and Zolt{\'a}n Szab{\'o}.
\newblock Homotopy {{\(K3\)}} surfaces and {{\(\text{mod }2\)}}
  {Seiberg}-{Witten} invariants.
\newblock {\em Math. Res. Lett.}, 4(1):17--21, 1997.

\bibitem[OP23]{orson_powell_2023}
Patrick Orson and Mark Powell.
\newblock Mapping class groups of simply connected {{\(4\)}}-manifolds with
  boundary.
\newblock Preprint available at https://arxiv.org/pdf/2207.05986.pdf, 2023.

\bibitem[OS04]{ozbagci_stipsicz_boook}
Burak \"{O}zba\u{g}c{\i} and Andr\'{a}s~I. Stipsicz.
\newblock {\em Surgery on contact 3-manifolds and {S}tein surfaces}, volume~13
  of {\em Bolyai Society Mathematical Studies}.
\newblock Springer-Verlag, Berlin; J\'{a}nos Bolyai Mathematical Society,
  Budapest, 2004.

\bibitem[OS06]{OzsvathSzaboHolomorphicTriangles}
Peter Ozsváth and Zoltán Szabó.
\newblock Holomorphic triangles and invariants for smooth four-manifolds.
\newblock {\em Adv. Math.}, 202(2):326--400, 2006.

\bibitem[Per86]{perron_1986}
Bernard Perron.
\newblock Pseudo-isotopies et isotopies en dimension quatre dans la
  cat\'{e}gorie topologique.
\newblock {\em Topology}, 25(4):381--397, 1986.

\bibitem[Qui86]{quinn_1986}
Frank~S. Quinn.
\newblock Isotopy of 4-manifolds.
\newblock {\em J. Differ. Geom.}, 24:343--372, 1986.

\bibitem[RS23]{ruberman_strle_2023}
Daniel Ruberman and Sa{\v{s}}o Strle.
\newblock Wall's stable realization for diffeomorphisms of definite
  {$4$}-manifolds.
\newblock Preprint available at https://arxiv.org/pdf/2210.16260.pdf, 2023.

\bibitem[Sae06]{saeki_2006}
Osamu Saeki.
\newblock Stable mapping class groups of {{\(4\)}}-manifolds with boundary.
\newblock {\em Trans. Am. Math. Soc.}, 358(5):2091--2104, 2006.

\bibitem[SS21]{SasahiraStoffregen}
Hirofumi {Sasahira} and Matthew {Stoffregen}.
\newblock {Seiberg-Witten Floer spectra for {$b_1>0$}}.
\newblock Preprint available at https://arxiv.org/pdf/2103.16536.pdf, 2021.

\bibitem[Tau94]{taubes_1994}
Clifford~H. Taubes.
\newblock The {Seiberg-Witten} invariants and symplectic forms.
\newblock {\em Math. Res. Lett.}, 1:809--822, 1994.

\bibitem[Wal64]{wall_1964}
C.~T.~C. Wall.
\newblock Diffeomorphisms of 4-manifolds.
\newblock {\em J. Lond. Math. Soc.}, 39:131--140, 1964.

\bibitem[Wal67]{WaldhausenZentrum}
Friedhelm Waldhausen.
\newblock Gruppen mit {Z}entrum und 3-dimensionale {M}annigfaltigkeiten.
\newblock {\em Topology}, 6(4):505--517, 1967.

\bibitem[Wal68]{Waldhausen68}
Friedhelm Waldhausen.
\newblock On irreducible 3-manifolds which are sufficiently large.
\newblock {\em Ann. Math.}, 87(1):56--88, 1968.

\bibitem[Wit94]{witten_1994}
Edward Witten.
\newblock {Monopoles and four manifolds}.
\newblock {\em Math. Res. Lett.}, 1:769--796, 1994.

\end{thebibliography}
\end{document}